\documentclass[12pt]{article}
\usepackage{amssymb,amsmath,amsfonts}
\usepackage{graphics}
\usepackage{graphicx}

\def\eqref#1{$(\ref{#1})$}

\newenvironment{proof}{\noindent {\em {Proof}}.}{$\square$
\medskip}
\newtheorem{theorem}{Theorem}[section]
\newtheorem{corollary}[theorem]{Corollary}
\newtheorem{lemma}[theorem]{Lemma}
\newtheorem{remark}[theorem]{Remark}
\newtheorem{proposition}[theorem]{Proposition}
\newtheorem{definition}[theorem]{Definition}
\newtheorem{example}[theorem]{Example}
\newtheorem{question}[theorem]{Question}
\newtheorem{problem}[theorem]{Problem}

\newcommand{\bt}[1]{\begin{theorem}\label{#1}}
\newcommand{\bc}[1]{\begin{corollary}\label{#1}}
\newcommand{\bl}[1]{\begin{lemma}\label{#1}}
\newcommand{\bp}[1]{\begin{proposition}\label{#1}}
\newcommand{\be}[1]{\begin{example}\rm\label{#1}}
\newcommand{\bq}[1]{\begin{question}\rm\label{#1}}
\newcommand{\bprob}[1]{\begin{problem}\rm\label{#1}}
\newcommand{\beq}[1]{\begin{eqnarray}\label{#1}}
\newcommand{\br}[1]{\begin{remark}\rm\label{#1}}

\newcommand{\el}{\end{lemma}}
\newcommand{\ep}{\end{proposition}}
\newcommand{\ee}{\end{example}}
\newcommand{\eq}{\end{question}}
\newcommand{\eprob}{\end{problem}}
\newcommand{\eeq}{\end{eqnarray}}
\newcommand{\ed}{\end{definition}}
\newcommand{\et}{\end{theorem}}
\newcommand{\ec}{\end{corollary}}
\newcommand{\er}{\end{remark}}

\title{\bf Proximal quasi-normal structure and existence of best proximity points}
\author{Farhad Fouladi, Ali Abkar\footnote{corresponding author}
\vspace{0.5in}\\
Department of Pure Mathematics, Faculty of Science,\\
Imam Khomeini International University, Qazvin 34149, Iran\\
\\
\\
\texttt{Email:f[undeline]folade@yahoo.com,~abkar@sci.ikiu.ac.ir}\\
 }
\date{}
\begin{document}
\maketitle \textbf{Abstract.} In this paper, we use the concept of proximal quasi-normal structure (P. Q-N. S) to study the existence of best proximity points for cyclic mappings, cyclic contractions, relatively Kannan nonexpansive mappings, as well as for orbitally
nonexpansive mappings. In this way, we generalize several recent results obtained by others.\\


\textbf{}{}
\section{Introduction}
Let $(X,d)$ be a metric space and $K\subseteq X$. A
mapping $T:K \rightarrow K$ is said to be nonexpansive if $d(Tu, tv)\leq
d(u, v)$ for all $u,v\in K$. Nonexpansive mappings are those which
have Lipschitz constant equal to one. For example, contractions,
isometries and resolvents of accretive operators on normed spaces are nonexpansive.
The solutions of the equation $Tu=u$ are fixed points of the mapping
$T:K\rightarrow K$. If $A,B$ are nonempty subsets of a metric space
$(X,d)$ and $T:A\cup B\rightarrow A\cup B$, then for the existence of a fixed
point it is necessary that $T(A)\cap A\neq \emptyset$. If this does not
hold, $d(u,Tu)>0$ for each $u\in A$. In this situation our aim is to
minimize the term $d(u,Tu)$. This line of investigation gives rise to the best approximation theory.\par
Assume that $A,B$ are nonempty subsets of a metric space $X$. A mapping
$T:A\cup B\rightarrow A\cup B$ is said to be cyclic provided that
$T(A)\subseteq B$ and $T(B)\subseteq A$. We say that the pair $(A,B)$ of nonempty subsets
in a metric space $X$ satisfies a property if both $A$ and $B$
satisfy that property. For example, $(A, B)$ is convex if and only
if both $A$ and $B$ are convex. Moreover, throughout this paper we
shall use the following notations and definitions:
\begin{equation*}\label{eq3}
\begin{array}{lll}
(A,B)\subseteq &(C,D) &\Leftrightarrow A\subseteq C \hspace{ .5 cm} B\subseteq D,\\
&\delta_u(A)&=\sup\{d(u,v):v\in A\} \hspace{.5 cm} \forall  u\in X,\\
&\delta(A,B)&=\sup\{d(u,v): u\in A,v\in B\}.\\
\end{array}
\end{equation*}
The notion of normal structure for Banach spaces was introduced by
Brodskii and Milman in \cite{4}, where it was shown that every
weakly compact convex set which has this property contains a point
which is fixed under surjective isometry. More information on normal
structure, can be found in \cite{11, 12,13,18}. Abkar and Gabeleh in \cite{7} introduced the notion of
proximal quasi-normal structure (P. Q-N. S) and proved some fixed point problems under this structure.
In this article we aim to generalize some problems already proved under stronger assumptions in \cite{5,8,19}.
The main objective here is the assumption of having proximal quasi normal structure, instead of having proximal normal structure.
In all this cases we shall prove the existence of best proximity (or fixed) points for the class of mappings under discussion. \par
In section 2 we study cyclic contractions and by recalling some results of M. Gabeleh \cite{9} for cyclic orbitally contraction mappings, and using the convex structure for metric spaces, we
prove the existence of best proximity points for cyclic contractions in Banach spaces with(P. Q-N. S).
In section 3 we recall some definitions from \cite{8} and prove the existence of best proximity points for relatively Kannan nonexpansive and cyclic contraction mappings with the proximally compact structure as introduced by M. Gabeleh \cite{8}. 
In section 4, we take up the class of cyclic relatively nonexpansive mappings, and finally in the last section we consider
orbitally nonexpansive mappings presented in \cite{8} by L-F. Enrique, in this case we replace the structure in the main result of \cite{8} and prove the existence of fixed point for this mappings with (P. Q-N. S).
\section{Cyclic contraction mappings}
In \cite{3}, Eldered and Veermani introduced the class of cyclic
contractions. For this class of mappings, they proved best proximity point theorems. Now, we study
the same problem in the situation in which the metric space under discussion has a convex structure (see below for the definition).
\begin{definition}\label{def_4_1}\cite{3}
Let $A$ and $B$ be nonempty subsets of a metric space $X$. A mapping
$T:A\cup B\rightarrow A\cup B$ is said to be a cyclic contraction if
$T$ is cyclic and
$$d(Tu,Tv)\leq k d(u,v)+(1-k)dist(A,B)$$
for some $k\in(0,1)$ and for all $u\in A, v\in B$.
\end{definition}
Let $T$ be a cyclic mapping. A point $u\in  A\cup B$ is said to be a
best proximity point for $T$ provided that $d(u,Tu)=dist(A,B)$. For
a uniformly convex Banach space $X$, Eldered and Veermani proved the
following theorem:
\begin{theorem}\label{th_4_1}
\cite{1} Let $A$ and $B$ be nonempty closed convex subsets of
uniformly convex Banach space $X$ and let $T:A\cup B\rightarrow
A\cup B$ be a cyclic contraction map. For $u_0\in A$ define
$u_{n+1}:=Tu_n$ for each $n\geq 0$. Then there exists a unique $u\in A$ such
 that $u_{2n}\rightarrow u$ and $\|u-Tu\|=dist(A,B)$.
\end{theorem}
In \cite{19}, Takahashi introduced the notion of convexity in metric
spaces as follows:
\begin{definition}\label{def4_2}\cite{19}
Let $(X,d)$ be a metric space and $J:=[0,1]$. A mapping $X\times
X\times J\rightarrow X$ is said to be a convex structure on $X$
provided that for each $(u,v;\lambda)\in X\times X\times
J\rightarrow X$ and $u\in X$,
$$d(z,w(u,v;\lambda))\leq \lambda d(z,u)+(1-\lambda) d(z,v).$$
\end{definition}
A metric space (X,d) together with a convex structure $W$ is called
a convex metric space, which is denoted by $(X,d,W)$. A Banach space
and each of its convex subsets are convex metric spaces.

To describe our results, we need some definitions and preliminary
facts from the reference \cite{19}.

Notice that each closed ball in a convex metric space is a convex subset of that space.

\begin{definition}\label{def4_4}\cite{19}
A convex metric space $(X,d,W)$ is said to have property (C) if
every bounded decreasing net of nonempty closed subsets of $X$ has a
nonempty intersection.
\end{definition}
For example every weakly compact convex subset of a Banach space has
property (C).

Now we prove the existence of best proximity point for cyclic contraction mappings in convex metric space as
follows.

\begin{theorem}\label{th4_2}
Let $(A,B)$ be a nonempty, bounded, closed and convex pair in a
convex metric space $(X,d,W)$. Suppose that $T:A\cup B\rightarrow
A\cup B$ is a cyclic contraction. If $X$ has the property (C), then
$T$ has a best proximity point.
\end{theorem}
\begin{proof}
Let $\Omega$ denote the set of all nonempty, bounded, closed and
convex pairs $(E,F)$ which are subsets of $(A,B)$ and such that $T$
is cyclic on $E\cup F$. Note that $(A,B)\in \Omega$. Also, $\Omega$ is
partially ordered by the reverse inclusion, that is $(E_1,F_1)\leq
(E_2,F_2)\Leftrightarrow (E_2,F_2)\subseteq (E_1,F_1)$. By the fact
that $X$ has property (C) every increasing chain in $\Omega$ is
bounded above. So by using Zorn's lemma we obtain a minimal element
say $(K_1,K_2)\in \Omega$. We note that
$\left(\overline{con}(T(K_2)),\overline{co}(T(K_1))\right)$ is a
nonempty, bounded, closed and convex pair in $X$ and
$\left(\overline{con}(T(K_2)),\overline{co}(T(K_1))\right)\subseteq(K_1,K_2)$.
Further,
$$T(\overline{con}(T(K_2)))\subseteq T(K_1)\subseteq
\overline{con}(T(K_1)),$$ and also,
$$T(\overline{con}(T(K_1)))\subseteq \overline{con}(T(K_2))),$$
that is, $T$ is cyclic on $\overline{con}(T(K_2))\cup
\overline{con}(T(K_1))$. It now follows from the minimality of $(K_1,K_2)$
that
$$\overline{con}(T(K_2))=K_1,\hspace{.2 cm} \overline{con}(T(K_1))=K_2,$$
Let $u\in K_1$, then $K_2\subseteq B(u,\delta_u(K_2))$. Now, if $v\in K_2$
we have
$$d(Tu,Tv)\leq k d(u,v)+(1-k) dist(A,B)\leq k\delta(K_1,K_2)+(1-k)dist(A,B),\hspace{.2cm}k\in(0,1)$$
Hence, for all $v\in K_2$ we have
$$Tv\in B(Tu;k\delta(K_1,K_2)+(1-k)dist(A,B)),$$
and so,
$$T(K_2)\subseteq B(Tu;k\delta(K_1,K_2)+(1-k)dist(A,B)).$$
Thus,
$$K_1=\overline{con}(T(K_2))\subseteq B(Tu;k\delta(K_1,K_2)+(1-k)dist(A,B)).$$
Therefore
$$d(z,Tu)\leq k\delta(K_1,K_2)+(1-k)dist(A,B),\hspace{.5 cm} \forall z\in K_1.$$
This implies that
\begin{equation}\label{star_1}
\delta_{Tu}(K_1)\leq k\delta(K_1,K_2)+(1-k)dist(A,B).
\end{equation}
Similarly, if $v\in K_2$, we conclude that
\begin{equation}\label{star_2}
\delta_{Tv}(K_2)\leq k\delta(K_1,K_1)+(1-k)dist(A,B).
\end{equation}
If we put
$$L_1:=\{u\in K_1; \delta_u(K_2)\leq k \delta(K_1,K_2)+(1-k)dist(A,B)\},$$
$$L_2:=\{v\in K_2; \delta_v(K_1)\leq k \delta(K_1,K_2)+(1-k)dist(A,B)\},$$
then $T(K_2)\subseteq L_1$  and $T(K_1)\subseteq L_2$. Beside, it is easy to
check that
$$L_1=\bigcap_{v\in K_2} B(v;k\delta(K_1,K_2)+(1-k)dist(A,B)) \bigcap K_1,$$
$$L_2=\bigcap_{u\in K_1} B(u;k\delta(K_1,K_2)+(1-k)dist(A,B)) \bigcap K_2.$$
Further, if $u \in L_1$ then by \eqref{star_1}, $Tu\in L_2$, i.e.
$T(L_1)\subseteq L_2$ and also by the relation \eqref{star_2}
$T(L_2)\subseteq L_1$, that is, $T$ is cyclic on $L_1\cup L_2$. It now
follows from the minimality of $(K_1,K_2)$ that $L_1=K_1$ and $L_2=K_2$. This
deduces that
$$\delta_u(K_2)\leq k \delta(K_1,K_2)+(1-k)dist(A,B),\hspace{.5 cm} \forall u\in K_1.$$
We have
$$\delta(K_1,K_2)=\sup_{u\in K_1}\delta_u(K_2)\leq k \delta(K_1,K_2)+(1-k)dist(A,B),$$
and then $\delta(K_1,K_2)=dist(A,B)$. Therefore for each pair
$(u^*,v^*)\in K_1\times K_2$ we must have
$$d(u^*,Tu^*)=d(Tv^*,v^*)=dist(A,B)$$
\end{proof}
\section{Relatively Kannan nonexpansive mappings}
In this section we investigate the existence of best proximity points
for relatively Kannan nonexpansive mappings in the setting of convex
metric spaces.
\begin{definition}\label{def4_5}\cite{11}
A mapping $T:K\rightarrow K$ is said to be Kannan nonexpansive,
provided that
$$d(Tu,Tv)\leq \frac{1}{2}\left[d(u,Tu)+d(v,Tv) \right],$$
for all $u,v \in K$. Notice that mappings of the above type may or
may not be nonexpansive in the usual case. In fact, the Kannan
condition does not even imply continuity of the mapping.
\end{definition}

\begin{definition}\label{def4_6}\cite{8}
 Let $(A,B)$ be a nonempty pair of subsets of a metric space
$(X,d)$. We say that the pair $(A,B)$ is proximally compact
provided that every net $(\{u_{\alpha}\},\{v_{\alpha}\})$ of
$A\times B$ satisfying the condition that
$d(u_{\alpha},v_{\alpha})\rightarrow dist(A,B)$, has a convergent
subnet in $A\times B$.
\end{definition}
It is clear that if $(K_1,K_2)$ is a compact pair in a metric space
$(X,d)$ then $(K_1,K_2)$ is proximally compact. Next we shall present the
main result of this section.

\begin{theorem}\label{th4_2}
Let $(A,B)$ be a nonempty, bounded, closed and convex pair in a
convex metric space $(X,d,W)$ such that $A_0$ is nonempty and
$(A,B)$ is proximally compact. Suppose that $T:A\cup B\rightarrow
A\cup B$ is a relatively Kannan nonexpansive mapping and $(A,B)$ has
the P.Q-N.S. Moreover, let $T$ be a
cyclic contraction and $X$ has the property (C). Then $T$ has a
best proximity point.
\end{theorem}
\begin{proof}
Let $\Omega$ denote the set of all nonempty, closed and convex pairs
$(	C,D)$ which are subsets of $(A,B)$ and such that $T$ is cyclic on
$C\cup D$ and $d(u,v) = dist(A,B)$ for some $(u,v)\in C\times D$.
Note that $(A,B)\in \Omega$ by the fact that $A_0$ is nonempty. Also,
$\Omega$ is partially ordered by the reverse inclusion. Assume that
$\{(C_{\alpha},D_{\alpha})\}_{\alpha}$ is a descending chain in
$\Omega$. Set $C:=\bigcap C_{\alpha}$ and $D:=\bigcap D_{\alpha}$.
Since $X$ has the property (C), we conclude that $(C,D)$ is a
nonempty pair. Also, by Proposition 2.4 of \cite{9}, $(C,D)$ is a closed
and convex pair. Moreover,
$$T(C)=T(\bigcap C_{\alpha})\subseteq \bigcap T( C_{\alpha})\subseteq \bigcap D_{\alpha}=D$$
Similarly, we can see that $T(D)\subseteq C$, that is, $T$ is cyclic
on $C\cup D$. Now, let $(u_{\alpha}, v_{\alpha})\in C_{\alpha}\times
D_{\alpha}$ be such that $d(u_{\alpha}, v_{\alpha}) = dist(A,B)$.
Since $(A,B)$ is proximally compact, $(u_{\alpha}, v_{\alpha})$
has a convergent subsequence say $(u_{\alpha_i} , v_{\alpha_i})$
such that $u_{\alpha_i}\rightarrow u\in A$ and
$v_{\alpha_i}\rightarrow v \in N$. Thus,
$$d(u,v)=\lim_{i}d(u_{\alpha_i},v_{\alpha_i})=dist(A,B).$$
Therefore, there exists an element $(u, v)\in C\times D$ such that
$d(u,v) = dist(A,B)$. Hence, every increasing chain in $\Omega$ is
bounded above with respect to reverse inclusion relation. Then by
using Zorn's Lemma we can get a minimal element say $(E_1,E_2)$. Let
$r> 0$ be such that $r \geq dist(A,B)$ and consider $(u^*, v^*) \in
E_1 \times E_2$ such that
$$d(u^*,v^*)=dist(A,B),\hspace{.2 cm} d(u^*,Tu^*)\leq r \hspace{.2 cm} and \hspace{.2 cm} d(Tv^*,v^*)\leq r.$$
Set,
$$E_1^r=\{u\in E_1:d(u,Tu)\leq r\},\hspace{.2 cm} E_2^r=\{u\in E_2:d(Tu,u)\leq r\},$$
and put $L_1^r:= \overline{con}(T(E_1^r )), L_2^r :=
\overline{con}(T(E_2^r ))$. We prove that $T$ is cyclic on
$L_1^r\cup L_2^r$. At first, we show that $L_1^r\subseteq E_2^r$.
Let $u\in L_1^r$. If $d(Tu,u) = dist(A,B)$, then $u \in E_2^r$. Let
$d(Tu, u)> dist(A,B)$. Put, $h:= sup\{d(Tz, Tu) : z \in E_1^r\}$. We
note that $B(Tu;h) \supseteq T(E_1^r)$, that is, $L_1^r \subseteq
B(Tu; h)$. Since $u\in L_1^r$, we have $d(Tu, u) \leq h$. It now
follows from the definition of $h$ that for each $\epsilon > 0$
there exists $z\in E_1^r$ such that $h-\epsilon \leq d(Tz, Tu)$.
Therefore,
$$d(Tu,u)-\epsilon\leq h-\epsilon\leq d(Tz,Tu)$$
$$\leq \frac{1}{2}[d(z,Tz)+d(Tu,u)]\leq \frac{1}{2}d(Tu,u)+\frac{1}{2}r.$$
Hence, $d(Tu, u) \leq r + 2\epsilon$ and so, $u \in E_2^r$. Thus,
$L_1^r \subseteq E_2^r$ . This implies that
$$T(L_1^r)\subseteq T(E_2^r)\subseteq \overline{con}(T(E_2^r))=L_2^r.$$
Similarly, we can see that $T(L_2^r)\subseteq L_1^r$, that is, $T$
is cyclic on $L_1^r\cup L_2^r$. We now claim that $\delta(L^r_1
,L^r_ 2 ) \leq r$. By using Lemma 5.1 we have
$$\delta(L^r_1,L^r_ 2 )=\delta (\overline{con}(T(E_1^r)),\overline{con}(T(E_2^r)))$$
$$=\delta ((T(E_1^r)),(T(E_2^r)))=\sup \{d(Tu,Tv):u \in E_1^r,v\in E_2^r\}$$
$$\leq \sup\{\frac{1}{2}[d(u,Tu)+d(Tv,v)]:u\in E_1^r,v\in E_2^r\}$$
$$\leq\frac{1}{2}[r+r]=r.$$
Besides, since $(u^*,v^*)\in E_1\times E_2$ such that $d(u^*, v^*) =
dist(A,B)$, for any $\theta> 0$, we have $d(u^* ,v^*) < \theta +
dist(A,B)$. By the fact that $T$ is cyclic contraction,
$$dist(A,B)\leq d(Tu^*,Tv^*)<\alpha d(u^*,v^*) +(1-\alpha) dist(A,B) = dist(A,B),$$
This implies that $d(Tu^*; Tv^*) =
dist(A,B)$ and so, $dist(L^r_2 ,L^r_ 1 ) = dist(A,B)$. Let
$$r_0:=\inf\{d(u,Tu):u\in E_1\cup E_2\}.$$
Then $r_0 \geq dist(A,B)$. Let $r_n$ be a nonnegative sequence such
that $r_n\downarrow r_0$. Thus,
$\left(\{L_1^{r_n}\},\{L_2^{r_n}\}\right)$ are descending sequences
of nonempty, bounded, closed and convex subsets of $(E_1, E_2)$. Since
$X$ has the property (C),
$$L_1^{r_0}=\bigcap_{n=1}^{\infty} L_1^{r_n}\neq \emptyset, \hspace{.2 cm} L_2^{r_0}=\bigcap_{n=1}^{\infty} L_2^{r_n}\neq \emptyset.$$
Moreover, by the preceding argument, $T : L_1^{r_0}\cup
L_2^{r_0}\rightarrow L_1^{r_0}\cup L_2^{r_0}$ is a cyclic mapping.
Further, since $dist(L_1^{r_n},L_1^{r_n} ) = dist(A,B)$ for all
$n\in N$, we deduce that $dist(L_2^{r_0},L_1^{r_0} ) = dist(A,B)$.
It now follows from the minimality of $(E_1,E_2)$ that $L_2^{r_0} =
E_1$ and $L_1^{r_0}=E_2$. Therefore, $d(u, Tu) \leq r_0$ for all
$u\in E_1 \cup E_2$. Assume that $r_0 > dist(A,B)$. Since the pair
$(A,B)$ has P.Q-N.S, there exists $(x, y)
\in E_1 \times E_2$ such that
$$d(x,v)<\delta(E_1,E_2)\leq r_0 \hspace{.2 cm}\&\hspace{.2 cm} d(u,y)<\delta(E_1,E_2)\leq r_0.$$
for all $(u, v) \in E_1 \times E_2$. Hence
$$d(x,Tx)<\delta(E_1,E_2)\leq r_0 \hspace{.2 cm}\&\hspace{.2 cm} d(Ty,y)<\delta(E_1,E_2)\leq r_0.$$
This is a contradiction, that is, $r_0 = dist(A,B)$ and so,
$$d(u,Tu)=d(v,Tv)=dist(A,B),$$
for all $(u, v) \in E_1 \times E_2$. This completes the proof.
\end{proof}
\section{Cyclic relatively nonexpansive mapping}
In this section we study the existence of best proximity points for cyclic relatively nonexpansive mappings in the sense that they are defined on the union of two subsets$A$ and $B$ of a Banach space $X$, $T(A)\subseteq B$ , $T(B)\subseteq A$ and satisfy $\|Tu-Tv\|\leq \|u-v\|$ for all $(u,v)\in
A\times B$.
Eldred et. al in \cite{6} proved the existence of best proximity points for relatively nonexpansive mappings with the proximal normal structure.
We shall obtain the same result under the weaker condition (P. Q-N. S).
\begin{definition}\label{def3.1}\cite{6}
A pair $(A,B)$ of subsets of a linear space $X$ is said to be a
proximal pair if for each $(u,v)\in A\times B$ there exists
$(\acute{u},\acute{v})\in A\times B$ such that
$$\|u-\acute{v}\|=\|\acute{u}-v\|=dist(A,B).$$
\end{definition}
We use $(A_0,B_0)$ to denote the proximal pair obtained from $(A,B)$ upon setting
$$A_0=\{u\in A :\|u-\acute{v}\|=dist(A,B)\, \text{for some}, \hspace{.1cm}  \acute{v} \in B\}\hspace*{.5cm}$$
and
$$B_0=\{v\in M :\|u-\acute{u}\|=dist(A,B)\, \text{for some}\, \hspace{.1cm}  \acute{u}\in A\}.$$

By using the geometric property of proximal normal structure Eldred et. al established the
following theorem. Before we mention the main theorem of \cite{6},
we recall the following definition.

\begin{definition}\label{def3.3}(\cite{6})
Let $(A,B)$ be a nonempty pair of a normed linear space $X$ and
$T:A\cup B\rightarrow A\cup B$ be a mapping. We say that $T$ is cyclic
relatively nonexpansive if $T(A)\subseteq B$ , $T(B)\subseteq A$ and  $\|Tu-Tv\|\leq \|u-v\|$ for all $(u,v)\in
A\times B$.
\end{definition}

\begin{theorem}\label{th3.4}
(\cite{6}) Let $(A,B)$ be a nonempty, weakly compact convex pair in
a Banach space $X$, and suppose $(A,B)$ has proximal normal
structure. Assume that $T:A\cup B\rightarrow A\cup B$ is a cyclic
relatively nonexpansive mapping. Then $T$ has a best proximity point
in both $A$ and $B$, that is, there exists $(u^*,v^*)\in A\times B$
such that $\|u^*-Tu^*\|=\|Tv^*-v^*\|=dist(A,B)$.
\end{theorem}
Before we state the main result of this section, we recall the
following lemma.
\begin{lemma}\label{lem3.7}\cite{7}
Let $(F_1,F_2)$ be a nonempty pair of a normed linear space $X$.
Then
$$\delta(F_1,F_2)=\delta(\overline{co}(F_1),\overline{co}(F_2))$$.
\end{lemma}
Now we prove the main result of this section.
\begin{theorem}\label{th4.2}
Let $(A,B)$  be a nonempty, weakly compact convex pair in
a Banach space $X$ and
suppose $(A,B)$ has P.Q-N.S. Assume that
$T:A\cup B \rightarrow A\cup B$ is a cyclic relatively nonexpansive
mapping. Then $T$ has a best proximity point in both $A$ and $B$,
that is, there exists $(u^*,v^*)\in A\times B$ such that
$\|u^*-Tu^*\|=\|Tv^*-v^*\|=dist(A,B)$.
\end{theorem}
\begin{proof}
It is not difficult to see that $(A_0,B_0)$ is a nonempty, weakly
compact convex pair and $dist(A,B)=dist(A_0,B_0)$. Moreover, $T$ is
cyclic on $A_0\cup B_0$ (for details, see \cite{6}). Let $\Omega$
denote the collection of all nonempty, weakly compact convex pairs
$(E,F)$ which are subsets of $(A,B)$ and $dist(E,F)=dist(A,B)$ and
$T$ is cyclic on $E\cup F$. Then $\Omega$ is nonempty, since
$(A_0,B_0)\in\Omega$. By using Zorn's lemma we can see that $\Omega$ has
a minimal element, say $(K_1,K_2)$ with respect to revers inclusion
relation and $dist(K_1,K_2)=dist(A,B)$. Now let $r$ be a real positive number such that $r\geq
dist(A,B)$ and let $(u,v)\in K_1\times K_2$ is such that
\begin{equation*}\label{eq_n_1}
\|u-v\|=dist(A,B),\hspace{.5 cm} \|u-Tu\|\leq r,
\end{equation*}
and $\|Tv-v\|\leq r$. Put
\begin{equation*}\label{eq_n_2}
K_1^r=\{x\in K_1: \|x-Tx\|\leq r\},\hspace{.5 cm} K_2^r=\{x\in K_2:
\|Tx-x\|\leq r\},
\end{equation*}
and set $L_1^r:=\overline{co}(T(K_1^r)))$,
$L_2^r:=\overline{co}(T(K_2^r)))$. We claim that $T$ is cyclic on
$L_1^r\cup L_2^r$. Firstly, We prove $L_1^r \subseteq K_2^r$. Let
$x\in L_1^r$. If $\|Tx-x\|=dist(A,B)$, then $x\in K_2^r$. Suppose
$\|Tx-x\|>dist(A,B)$. Set $S:=\sup\{\|Tw-Tx\|;w\in K_1^r\}$. Then
$T(k_1^r)\subseteq B(Tx;S)$. This implies that $L_1^r\subseteq
B(Tx;S)$. Since $x\in L_1^r$, we have $\|Tx-x\|\leq S$. By the
definition of $S$, for each $\epsilon >0$ there exists $w\in K_1^r$
such that $S-\epsilon\leq \|Tw-Tx\|$. Hence
\begin{equation*}\label{eq_n_3}
\|Tx-x\|-\epsilon \leq S-\epsilon\leq \|Tw-Tx\|\leq \|w-x\|\leq r.
\end{equation*}
This implies that $\|Tx-x\|\leq r-\epsilon$. Thus $x\in K_2^r$ and
hence $L_1^r\subseteq K_2^r$. Therefore
\begin{equation*}\label{eq_n_4}
T(L_1^r)\subseteq T(K_2^r)\subseteq \overline{co}(T(K_2^r))=L_2^r.
\end{equation*}
Similar argument shows that $T(L_2^r)\subseteq L_1^r$. Thus $T$ is
cyclic on $L_1^r\cup L_2^r$. We now prove that
$\delta(L_1^r,L_2^r)\leq r$. By Lemma~\ref{lem3.7} we have
\begin{equation*}\label{eq_n_5}
\begin{array}{ll}
\delta(L_1^r,L_2^r)&=\delta\left(\overline{co}(T(K_1^r)),T(K_2^r)\right)\\
&=\delta(T(K_1^r),T(K_2^r))\\
&=\sup\{\|Tx-Ty\|;x\in K_1^r,y\in K_2^r\}\\
&\leq \sup\{\|x-y\|;x\in K_1^r, y\in K_2^r\}\\
&\leq r.
\end{array}
\end{equation*}
On the other hand since $u\in K_1^r$, $v \in K_2^r$ and
$\|u-v\|=dist(A,B)$, we conclude that
\begin{equation*}\label{eq_n_6}
dist(A,B)\leq dist(L_2^r,L_1^r)\leq \|Tv-Tu\|\leq \|u-v\|=dist(A,B),
\end{equation*}
that is; $dist(L_2^r,L_1^r)=dist(A,B)$.\\
Put
\begin{equation*}\label{eq_n_7}
r_0=\inf\{\|x-Tx\|;x\in K_1\cup K_2\}.
\end{equation*}
Then $r_0\geq dist(A,B)$. Let $\{r_n\}$ be a nonnegative sequence
such that $r_n\downarrow r_0$. Thus $\{L_1^{r_n}\}$, $\{L_2^{r_n}\}$
are descending sequence of nonempty, weakly compact convex subsets
of $K_1, K_2$ respectively. By the weakly compactness of $K_1$ and
$K_2$ we must have
\begin{equation*}\label{eq_n_8}
L_1^{r_0}=\bigcap_{n=1}^{\infty}L_1^{r_n}\neq \emptyset, \hspace{.5
cm} L_2^{r_0}=\bigcap_{n=1}^{\infty}L_2^{r_n}\neq \emptyset.
\end{equation*}
Also by the preceding argument $T:L_1^{r_0}\cup L_2^{r_0}\rightarrow
L_1^{r_0}\cup L_2^{r_0}$ is a cyclic mapping. Moreover, Since
$dist(L_2^{r_n},L_1^{r_n})=dist(A,B)$ for all $n\in \mathbb{N}$, we
conclude that $dist(L_2^{r_0},L_1^{r_0})=dist(A,B)$. The minimality
of $(K_1,K_2)$ implies that $L_2^{r_0}=K_1$ and $L_1^{r_0}=K_2$.
Hence $\|x-Tx\|\leq r_0$ for all $x \in K_1 \cup K_2$. Now let
$r_0>dist(A,B)$. Since $(A,B)$ has P.Q-N.S,
there exists $(u_1,v_1)\in K_1 \times K_2$ such that
\begin{equation*}\label{eq_n_9}
\|u_1-y\|<\delta(K_1,K_2)\leq r_0,\hspace{.5 cm}
\|x-v_1\|<\delta(K_1,K_2)\leq r_0,
\end{equation*}
for all $(x,y)\in K_1 \times K_2$. Therefore
\begin{equation*}\label{eq_n_10}
\|u_1-Tu_1\|<\delta(K_1,K_2)\leq r_0\hspace{.5 cm}  \& \hspace{.5
cm} \|Tv_1-v_1\|<\delta(K_1,K_2)\leq r_0,
\end{equation*}
which is a contradiction. This implies that $r_0=dist(A,B)$ and
hence
\begin{equation*}\label{eq_n_11}
\|x-Tx\|=\|Ty-y\|=dist(A,B),
\end{equation*}
for all $(x,y)\in K_1 \times K_2$.
\end{proof}

\begin{example}\label{ex_n_1}
Let $X=\mathbb{R}$ with the usual metric and $A=[0,2]$ and
$B=[3,4]$. Define $T:A\cup B \rightarrow A\cup B$ with
\begin{equation*}\label{eq_n_12}
T(x)=\left\{\begin{array}{lll}
4&if&x=0,\\
3&if&x\neq 0,x\in A,\\
2&if&x\in B.\\ \end{array} \right.
\end{equation*}
It is easy to check that $T$ is cyclic on $A\cup B$ and be a
relatively nonexpansive mapping. Thus by theorem~\ref{th4.2}, $T$
has a best proximity point in $A\cup B$.
\end{example}

\section{Orbitally nonexpansive mapping}
Several definitions of generalized nonexpansive mappings are
concerned with the iterates of the mapping under consideration and
hence they are related to the behavior of its orbits. For instance,
a mapping $T:K\rightarrow K$ is said to be asymptotically
nonexpansive (see \cite{10}) if for all $u,v \in K,
\|T^n(u)-T^n(v)\leq k_n\|u-v\|$, where $(k_n)$ is a sequence of real
numbers such that $\lim k_n=1$.
Enrique. L in \cite{5} introduced the orbitally nonexpansive mappings which some of orbits behave as a kind of attractor. In this section, we recall the definition of these mappings and propositions of \cite{5}. We consider the existence of a fixed point for this category of mappings. Enrique. L in \cite{5} has examined the existence of a fixed point for orbitally nonexpansive mappings with the P.N.S. Here, we want here to have a fixed point for this class of mappings under a weaker condition called P.Q-N.S.

\begin{definition}\label{def3.2}\cite{6}
A convex pair $(K_1,K_2)$ in a Banach space $X$ is said to have
proximal normal structure (P.N.S) if for any bounded, closed and convex
proximal pair $(H_1,H_2)\subseteq (K_1,K_2)$ for which
$dist(H_1,H_2)=dist(K_1,K_2)$ and $\delta(H_1,H_2)>dist(H_1,H_2)$,
there exits $(u_1,u_2)\in H_1\times H_2$ such that
$$\delta_{u_1}(H_2)<\delta(H_1,H_2),\hspace{.5 cm} \delta_{u_2}(H_1)<\delta(H_1,H_2).$$
\end{definition}
We now recall the notion of proximal quasi-normal structure from \cite{7}.
\begin{definition}\label{def3.5} \cite{7}
A convex pair $(K_1,K_2)$ in a Banach space $X$ is said to have
proximal quasi-normal structure (P.Q-N.S) if for any bounded, closed and
convex proximal pair $(H_1,H_2)\subseteq (K_1,K_2)$ for which
$dist(H_1,H_2)=dist(K_1,K_2)$ and $\delta(H_1,H_2)>dist(H_1,H_2)$,
there exits $(p_1,p_2)\in H_1 \times H_2$ such that
$$d(p_1,v)<\delta(H_1,H_2), \hspace{.5 cm} d(u,p_2)<\delta(H_1,H_2),$$
for all $(u,v)\in H_1\times H_2$.
\end{definition}
It follows from Definition 2.2 that for a convex subset $K$ of a
Banach space $X$, the pair $(K, K)$ has P.Q-N.S if and only if $K$ has quasi-normal structure. Moreover,
$$P.N.S \Rightarrow P.Q-N.S.$$
Notice that each pair of convex, closed and bounded pairs in uniformly convex Banach space has P.Q-N.S. \cite{7}
\begin{example}\label{ex3.6}
Let $(A,B)$ be a compact pair in a
Banach space $X$. Then $(A,B)$ has P.Q-N.S.
\end{example}
\begin{proof}
By Proposition 2.2 of \cite{6}, $(A,B)$ has P.N.S and thus has P.Q-N.S.
\end{proof}

\begin{definition}\label{def3.1_6}\cite{5}
Let $(X,\|.\|)$is a real Banach space and $C$ is nonempty, closed, convex and bounded subset
of $X$. A mapping $T:C\rightarrow C$ is said be orbitally nonexpansive if
for every nonempty, closed, convex, T-invariant subset $D$ of $C$,
there exists some $u_0\in D$ such that for every $u\in D$,
\begin{equation}\label{eq3.1}
\limsup_{n\rightarrow \infty}\|T^n(u_0)-T(u)\|\leq
\limsup_{n\rightarrow \infty}\|T^n(u_0)-u\|
\end{equation}
\end{definition}
By the above definition, it is clear that each nonexpansive mapping is orbitally nonexpansive.\cite{5}

Notice that orbitally nonexpansive mapping need not be continuous as
(\cite{17}, Example 1.1) and (\cite{14}, Proposition 3.4).
\begin{example}\label{ex3.4}
Let $T:[0,1]\rightarrow [0,1]$ be given by $T(u)=u^2$. The
closed convex $T$-invariant subsets of $[0, 1]$ are just all the
intervals of the form $[a,1]$ with $a\in [0,1]$. Choosing $u_0=0$
gives $T^n(u_0)=0$. Therefore, for every $u\in[a,1]$,
$$\limsup_{n\rightarrow \infty}|T^n(u_0)-T(u)|=|0-u^2|\leq |0-u|=\limsup_{n\rightarrow \infty}|T^n(u_0)-u|$$
and $T$ is an orbitally nonexpansive mapping.
\end{example}
Now we prove the main result of this section.

\begin{theorem}\label{th4.1}
Let $C$ be a nonempty weakly compact convex subset of a Banach space
$(X,\|.\|)$. Let
$T:C\rightarrow C$ be an orbitally nonexpansive mapping and $(C,C)$ has P.Q-N.S Then $T$
has a fixed point.
\end{theorem}

\begin{proof}
Since $C$ is a weakly compact set, from a standard application of
Zorn's lemma, there is a nonempty, closed, convex and T-invariant
subset $D$ of $C$ with no proper subsets enjoying these
characteristics. From the definition of an orbitally nonexpansive
mapping, there exists $u_0\in D$ such that, for every $u\in D$,
$$\limsup_{n\rightarrow \infty}\|T^n(u_0)-T(u)\|\leq \limsup_{n\rightarrow \infty}\|T^n(u_0)-u\|.$$
We will distinguish two cases.\\
\textbf{Case I.} There exists $u_0\in D$ such that $T^n(u_0)=z$ for
$B$ large enough. We claim that, in this case, $z$ is a fixed point
of $T$. Indeed,
$$\|z-T(z)\|=\limsup_{n\rightarrow \infty}\|T^n(u_0)-T(z)\|\leq \limsup_{n\rightarrow \infty}\|T^n(u_0)-z\|=0.$$
\textbf{Case II.} The sequence $(T^n(u_0))$ is bounded and not
(eventually) constant. Since the pair $(C,C)$ has
P.Q-N.S, there exists $(p,p_1)\in D\times D$ such that,
for every $v\in D$; $d(p,v)<diam(D)$
and for every $u\in D$; $d(u,p_1)<diam(D)$.
Since the sequence
$(T^n(u_0))$ is not constant and by P.Q-N.S of $(C,C)$ and by
choosing the right $\alpha$  we have
$$\exists\hspace{.1 cm} m,k \in \mathbb{N} ,m\neq k\hspace{.2 cm}s.t. \hspace{.2 cm} 0<\alpha\|p-T^k(u_0)\|<\|p-T^m(u_0)\|.$$
By supposing $r:=\alpha\|p-T^k(u_0)\|$ we have
$$0\leq r<\|p-T^m(u_0)\|$$
$$\leq \limsup_{n\rightarrow \infty} \|p-T^n(u_0)\|:=g(p)\leq diam(D).$$
This results
$$\exists r>0 \hspace{.2 cm}s.t.\hspace{.2 cm}\exists v\in C \hspace{.2 cm} d(p,v)<r<diam(D), \hspace{.2 cm} g(p)>r.$$
Now we suppose $D_r:=\{u\in D ; g(u)\leq r\}$. We have
$$\forall u\in D_r;\hspace{.2 cm}g(Tu)\leq g(u)\leq r \Rightarrow Tu\in D_r \Rightarrow T:D_r\rightarrow D_r,$$
this shows that $D_r$ is T-invariant,\\
and
$$\forall u,v \in D_r,\hspace{.2 cm} \forall \lambda\in[0,1]; \hspace{.2 cm} g(\lambda u+(1-\lambda)v)=\limsup_{n\rightarrow \infty}\|(\lambda u+(1-\lambda)v)-T^n(u_0)\|$$
$$\leq \lambda \limsup_{n\rightarrow \infty}\|u-T^n(u_0)\|+(1-\lambda)\limsup_{n\rightarrow \infty}\|v-T^n(u_0)\|$$
$$\leq \lambda r+(1-\lambda)r=r,$$
this means that $D_r$ is convex,\\
and\\
$$if\hspace{.2 cm} u_0\in \overline{D}_r \Rightarrow \exists \{v_n\}\subseteq D_r\hspace{.2 cm}s.t.\hspace{.2 cm} v_n\rightarrow u_0\Rightarrow g(u_0)=g(\lim_{n\rightarrow \infty}v_n)$$
$$\lim_{n\rightarrow \infty}g(v_n)\leq \lim_{n\rightarrow \infty}r=r$$
$$\Rightarrow u_0\in D_r \rightarrow \overline{D}_r\subseteq D_r,$$
thus $D_r$ is closed.\\
On the other hand since $D_r\subseteq D$ and $g(p)>r$ then there
exists $p\in D$ such that $p\notin D_r$ thus $D_r \subsetneqq D$.
Which is a contradiction to the minimality of $D$. Thus, case II is
not possible.
\end{proof}

\bibliographystyle{amsplain}

\begin{thebibliography}{99}

\bibitem{1} A. Abkar,  M. Gabeleh, \textit{Best proximity points for asymptotic cyclic contraction mappings}, Nonlinear Analysis.
 \textbf{74} (2011). 7261-7268.

\bibitem{2} A. Abkar and M. Gabeleh, \textit{Generalized cyclic contractions
in partially ordered metric space}, Optim. Lett., In press, DOI
10.1007/s11590-011-0379-y.

\bibitem{3} A. Anthony Eldred, P. Veeramani, \textit{Existence and convergence of best proximity points},
J. Math. Anal. Appl \textbf{323} (2006). 1001-1006.

\bibitem{4} M. S. Brodskii and D. P. Milman,\textit{ On the center of a
convex set}, Dokl. Akad. Nauk. SSSR. \textbf{59} (1948), 837--840.

\bibitem{5} L. Enrique, \textit{Orbitally nonexpansive mappings},
Bull. Aust. Math. Soc,  \textbf{93} (2016).497 – 503.

\bibitem{6} A. A. Eldred, W. A. Krik and P. Veeramani,\textit{ Proximal
normal structure and relatively  nonexpansive mappings}, Studia
Math. \textbf{171} (2005), 283--293.
\bibitem{7} M. Gabeleh and  A. Abkar, \textit{Proximal quasi-normal structure and a best proximity point theorem},
J. Nonlin. Convex Anal, Seria Matematica. \textbf{Volume 14. Number 4} (2013). 653-659.
\bibitem{8} M. Gabeleh, \textit{Minimal sets of noncyclic relatively nonexpansive mappings
in convex metric spaces},
Fixed Point Theory  to appear. \textbf{16(2015), No. 2, 313-322}.

\bibitem{9} M. Gabeleh, \textit{Proximal quasi-normal structure in convex metric spaces},
Analele Stiintifice ale Universitatii Ovidius Constanta, Seria Matematica. \textbf{22(3)} (2014). 45-58.

\bibitem{10} K. Goebel and W. A. Krik, \textit{A fixed point theorem for
asymptotically nonexpansive mapping}, Proc. Amer. Math. Soc.
\textbf{35} (1972), 171--174.

\bibitem{11} M. A. Khamsi, \textit{On metric spaces with uniform normal
structure}, Proc. Amer. Math. Soc. \textbf{106} (1989), 723--726.

\bibitem{12} W. A. Krik, \textit{A fixed point theorem for mappings which do
not increase distances}, Amer. Math. Monthly. \textbf{72} (1965),
1004--1006.

\bibitem{13} W. A. Krik,\textit{ Nonexpansive mapping and normal structure in
Banach spaces, Proceeding of the research workshop on Banach space
theory}, B. L. Lin (ed.), University of Iowa, 1981.

\bibitem{14} E. Llorens--Fuster and E. Moreno G$\acute{a}$lvez, \textit{The fixed point theory for some
generalized nonexpansive mappings}, Abstr. Appl. Anal.
\textbf{2011}(2011), Article ID 435686, 15 pages.,
doi:10.1155/2011/435686.

\bibitem{15} M. A. Petric, \textit{Best proximity point theorems for weak
cyclic kannan contractions}, Filomat, 25:1, DOI 10.2298/FIL1101145p
(2011)
\bibitem{16} P. M. Spardi, \textit{Struttura quasi normale e teoremi di
punto unito}. Rend. 1st Mat. Univ. Triest. \textbf{4} (1972),
105-114.
\bibitem{17} T. Suzuki, \textit{Fixed point theorem and convergence theorem
for some generalized nonexpansive mappings}, J. Math. Anal. Appl.
\textbf{340} (2008), 1088--1095.

\bibitem{18} S. Swaminathan, \textit{Normal structures in Banach spaces and
its generalizations, fixed points and nonexpansive mappings},
Contemp. Math. Amer. Math. Soc. \textbf{18} (1983). 201--215.
\bibitem{19}  W. Takahashi, \textit{A  convexity  in  metric  space  and  nonexpansive  mappings
}, KODAI MATH. SEM. REP, \textbf{22} (1970).  142-149.

\bibitem{20} C. Wong, \textit{Close-to-normal structure and its
applications}, J. Funct. Anal. \textbf{16} (1974). 353-358.


\end{thebibliography}

\end{document}